\documentclass[11pt]{article}
\usepackage{amsfonts}
\usepackage{amssymb}
\usepackage{mathrsfs}
\usepackage{srcltx}
\usepackage{subfigure}
\usepackage{amsmath}
\allowdisplaybreaks[4]

\usepackage{amsmath,amsthm,amstext,amsopn,graphicx,cite,xcolor,enumitem}
\newtheorem{theorem}{Theorem}[section]
\newtheorem{lemma}[theorem]{Lemma}
\newtheorem{proposition}[theorem]{Proposition}

\newtheorem{remark}[theorem]{Remark}

\theoremstyle{definition}

\theoremstyle{remark}

\numberwithin{equation}{section}
\newcommand\bes{\begin{eqnarray}}
\newcommand\ees{\end{eqnarray}}
\newcommand{\bess}{\begin{eqnarray*}}
\newcommand{\eess}{\end{eqnarray*}}
\newcommand{\ba}{\begin{array}}
\newcommand{\ea}{\end{array}}
\newcommand{\f}{\frac}

\newcommand{\Om}{\Omega}

\newcommand{\lf}{\left}
\newcommand{\rr}{\right}
\newcommand{\R}{{\mathbb R}}
\newcommand{\ds}{\displaystyle}
\newcommand{\dd}{\displaystyle}

\newcommand{\td}{\tilde}

\newcommand\yy{\infty}
\newcommand\ttt{t\to\yy}

\newcommand{\ol}{\overline}

\begin{document}
 \pagestyle{myheadings}


\date{}
\title{ \bf\large{Global stability of spatially nonhomogeneous steady state solution in a diffusive Holling-Tanner predator-prey model}\footnote{Partially supported by NSF Grant DMS-1853598, and NSFC Grants 11771110.}}
\author{Wenjie Ni\textsuperscript{1}\footnote{Corresponding Author.},\ \ Junping Shi\textsuperscript{2},\ \ Mingxin Wang\textsuperscript{3}
\\
{\small \textsuperscript{1} School of Science and Technology, University of New England, Armidale, NSW 2351, Australia\hfill{\ }}\\
{\small \textsuperscript{2} Department of Mathematics, William \& Mary,  Williamsburg, VA, 23187-8795, USA.\hfill{\ }}\\
{\small \textsuperscript{3} School of Mathematics, Harbin Institute of Technology, Harbin, 150001, China. \setcounter{footnote}{-1}{}\footnote{{\it E-mails}: wni2@une.edu.au (W.-J. Ni), jxshix@wm.edu (J.-P. Shi),
			mxwang@hit.edu.cn (M.-X. Wang)}\hfill{\ }}\\}
\maketitle

\begin{abstract}
The global stability of the nonhomogeneous positive steady state solution to a diffusive Holling-Tanner predator-prey model  in a heterogeneous  environment is proved by using a newly constructed Lyapunov function and estimates of nonconstant steady state solutions.  The techniques developed here can be adapted for other spatially heterogeneous consumer-resource models.
\end{abstract}

\section {Introduction}
In this paper we study the global dynamics of the following diffusive Holling-Tanner predator-prey model in a heterogeneous environment:
\begin{equation}\label{sp3}
\begin{cases}
\ds u_t=d_1(x)\Delta u+u\lf(a(x)-u-\frac{b v}{1+{r} u}\rr), & x\in \Omega,\; t>0,\\
\ds v_t=d_2(x)\Delta v+\mu v\lf(1-\frac{v}{u}\rr), & x\in \Omega,\; t>0,\\
\ds \f{\partial u}{\partial\nu}=\ds \f{\partial v}{\partial\nu}=0,& x\in \partial\Omega,\;
t>0,\\
u(x,0)=u_0(x),\ v(x,0)=v_0(x), & x\in \Omega.
\end{cases}
\end{equation}
Here  $u(x,t)$ and $v(x,t)$ are the density functions of prey and predator respectively, and $\Om$ is a bounded domain in $\R^n$ with a smooth boundary $\partial \Omega$; a no-flux boundary condition is imposed on $\partial \Omega$ so that the ecosystem is closed to exterior environment. $d_1(x)$ and $d_2(x)$ are the spatially dependent diffusion coefficient functions of prey and predator respectively; $a(x)$ is the spatially heterogeneous resource function, and other parameters   $b$, $r$ and $\mu$ are assumed to be constants. The non-spatial version of \eqref{sp3} was introduced in
\cite{may1974stability,tanner1975} as one of prototypical mathematical models describing predator-prey interactions.

For the non-spatial ODE model corresponding to \eqref{sp3}, it is known that for certain parameter range the unique positive steady state is globally asymptotically stable, while in other parameter range a unique limit cycle exists \cite{hsu1995,hsu1998}. The spatial model \eqref{sp3} in a homogeneous environment (assuming $d_i,a$ are constants) was first studied in \cite{peng2005}. The global stability of the positive constant steady state solution for the homogeneous case was proved in  \cite{cs2012,pw2007,Qi2016} under different conditions on parameters, and spatiotemporal pattern formation  for the homogeneous system \eqref{sp3} was considered in \cite{Jiang2013}. When $r=0$ in \eqref{sp3}, the system becomes to the  Leslie-Gower predator-prey model. The global stability for that case including delay effect was investigated in \cite{csw2012,Qi2016-2} when $a$ and $d_i$ are constants, see also \cite{Du2004,Min2019,Ni2016,Ni2017} for related work.

We define the nonlinearities in \eqref{sp3} to be
\begin{equation}\label{fg}
\ds f(y,u,v):=u\lf(y-u-\frac{b v}{1+{r} u}\rr),\ \ g(u,v):=\mu v\lf(1-\frac{v}{u}\rr),
\end{equation}
and we also denote
\begin{equation}\label{aa}
  \bar a:=\dd\max_{\ol\Omega} a(x), \;\; \underline a:=\dd\min_{\ol\Omega} a(x).
\end{equation}
Our results on the global dynamics of \eqref{sp3} are as follows:
\begin{theorem}\label{thm1}
Let  $\mu>0$ and $r\geq 0$ be  constants. Suppose that  $a$, $d_i\in C^{\alpha}(\ol\Omega)$ for some $0<\alpha<1$,  and $a(x)>0$, $d_i(x)>0$  on $\ol\Omega$; and the initial  functions $u_0\in C(\ol\Omega)$ and $v_0\in C(\ol\Omega)$ satisfy $u_0\geq,\not\equiv 0$ and $v_0\geq,\not\equiv 0$ on $\ol\Omega$. Assume that
\begin{align}\label{2.2}
0<b<{\underline a}/{\bar a}:=\frac{\min_{\ol\Omega} a(x)}{\max_{\ol\Omega} a(x)}.
\end{align}
\begin{enumerate}
  \item[{\rm (i)}] There exists a unique positive solution $(\underline u_\yy,\bar u_\yy,\underline v_\yy,\bar v_\yy)$ to the system of equations:
	\begin{equation*}
	f(\bar a,\bar u_\yy,\underline v_\yy)=0,\ \ f(\underline a,\underline u_\yy,\bar v_\yy)=0,\quad
	g(\bar u_\yy,\bar v_\yy)=0,\ \ g(\underline u_\yy,\underline v_\yy)=0.
	\end{equation*}
  \item[{\rm (ii)}] Let $(u(x,t),v(x,t))$ be the positive solution of problem \eqref{sp3}. Then
 \bes\label{bound}
 \left\{\begin{array}{ll}
\ds \underline u_{\yy}\leq \liminf_{t\to\yy}u(x,t)\leq \limsup_{t\to\yy}u(x,t)\leq \bar u_{\yy},\\
\ds \underline v_{\yy}\leq \liminf_{t\to\yy}v(x,t)\leq \limsup_{t\to\yy}v(x,t)\leq \bar v_{\yy}.
  \end{array}\right.\ees
  \item[{\rm (iii)}] If in addition,
  \bes\label{4.6}
 b<(1+2r \underline u_\yy-r \underline a)\lf[\frac{\min d_1(x)\min d_2(x)}
{\max d_1(x)\max d_2(x)} \rr]^{1/2} \lf(\frac{\underline u_\yy}{\bar u_\yy}\rr)^{5/2},
\ees
then  the problem \eqref{sp3} has a unique  positive steady state solution  $(u_*,v_*)$,  and
$\ds \lim_{\ttt} u(x,t)=u_*(x)$ and $\ds \lim_{\ttt} v(x,t)=v_*(x)$ in $C^2(\ol\Omega)$.
\end{enumerate}
\end{theorem}

The global stability of the positive steady state solution of \eqref{sp3} in the heterogeneous environment in Theorem \ref{thm1} also holds for $r=0$ which is the Leslie-Gower predator-prey model, and in that case, the condition \eqref{4.6} is simplified to
	\begin{align*}
 b<\lf[\frac{\min d_1(x)\min d_2(x)}
{\max d_1(x)\max d_2(x)} \rr]^{1/2} \lf(\frac{\min a(x)-b\max a(x)}{\max a(x)-b\min a(x)}\rr)^{5/2}.
\end{align*}

If  $d_1$, $d_2$ and $a$ are all constants,  then the problem \eqref{sp3} admits a unique positive constant steady state which can be solved as
\begin{equation}\label{ss}
\begin{split}
&u_\yy=\underline u_\yy=\bar u_\yy=\underline v_\yy=\bar v_\yy=\frac{-(b+1-ar)+\sqrt{(b+1-ar)^2+4ar}}{2r},\ \ \ {\rm when}\ r>0,\\
&u_\yy=\underline u_\yy=\bar u_\yy=\underline v_\yy=\bar v_\yy=\frac{a}{1+b},\ \ \ {\rm when}\ r=0.
\end{split}
\end{equation}
Then the global stability of the constant steady state $(u_\yy,u_\yy)$ in Theorem \ref{thm1} holds under the assumption $b<1$, which is the earlier result of \cite{cs2012}. In this case \eqref{4.6} is not needed as part (ii) already implies the global stability. Part (iii) of Theorem \ref{thm1} shows the global stability in the heterogeneous environment, and the condition \eqref{4.6} depends on the level of heterogeneity of $a(x)$ and $d_i(x)$. Indeed \eqref{4.6} holds when $a$ is nearly constant or when $b$ is sufficiently small (see Remark \ref{end}). We also remark that the global stability of the positive constant steady state of \eqref{sp3} is proved in \cite{Qi2016,Qi2016-2} with weaker condition on $b$ but constant $a$ and  $d_1=d_2$.

The proof of global stability combines the upper-lower solution method used in \cite{cs2012,Qi2016} and a newly developed  Lyapunov  functional method. For spatially homogeneous case, the upper-lower solution method alone can prove the global stability of the constant steady state of \eqref{sp3}, but in the spatially heterogeneous case, it only proves that the solutions are attracted into a rectangle defined as in \eqref{bound}. The Lyapunov function we use inside the attraction zone takes the form
\begin{align*}
\int_{\Om}\int_{u_*(x)}^{u(x,t)} \frac{u_*(x)}{d_1(x)}\frac{s-u_*(x)}{s} \text{d}s\text{d}x+\int_{\Om}\int_{v_*(x)}^{v(x,t)} \frac{v_*(x)}{d_2(x)}\frac{s-v_*(x)}{s} \text{d}s\text{d}x
\end{align*}
where $(u_*(x),v_*(x))$ is a positive steady state solution of \eqref{sp3}. The form of the Lyapunov function when $d_i$ and $a$ are constants is well-known, and here we use a spatially heterogenous form with weigh functions $u_*(x)/d_1(x)$ and   $v_*(x)/d_2(x)$ which is first used in \cite{nsw1} for proving the global stability of positive steady state of diffusive Lotka-Volterra competition system in the heterogeneous environment. It turns out that the weight functions encode the spatial heterogeneity of the environment so a non-constant steady state is achieved asymptotically. The new Lyapunov function developed here may be a useful tool to explore more general diffusive predator-prey models in the nonhomogeneous environment \cite{Du2004,DuShi2007}.

\section {Proof of Main results } 
\subsection{Existence   of positive solutions}
In the subsection, we show the existence and uniqueness  of the solution to \eqref{sp3}, and the existence of positive solution to the corresponding steady state problem:
\begin{equation}\label{esp3}
\begin{cases}
-d_1(x)\Delta u=u\lf(a(x)-u-\ds\frac{b v}{1+{r} u}\rr), & x\in \Omega,\\
\ds -d_2(x)\Delta v=\mu v\lf(1-\frac{v}{u}\rr), & x\in \Omega,\\
\ds \f{\partial u}{\partial\nu}=\ds \f{\partial v}{\partial\nu}=0,& x\in \partial\Omega.
\end{cases}
\end{equation}

We recall that the system \eqref{sp3} is called to be uniformly persistent    (see, e.g., \cite[Page 390]{hw1989}) if all positive solutions satisfy $\dd\liminf_{\ttt}u(x,t)>0$  and $\dd\liminf_{\ttt}v(x,t)>0$ for all  $x\in \bar{\Om}$, and it is permanent (see, e.g., \cite{cc1993,hs1992}) if they also satisfy $\dd\limsup_{t\to\yy}u(x,t)\leq M$ and $\dd\limsup_{t\to\yy}v(x,t)\leq M$ for some $M>0$. The following result shows the basic dynamics of \eqref{sp3}.
\begin{proposition}\label{th2.1} Let $b>0$, $\mu>0$ and $r\geq 0$ be  constants. Suppose that  $a$, $d_i\in C^{\alpha}(\ol\Omega)$ for some $0<\alpha<1$,  and $a(x)>0$, $d_i(x)>0$  on $\ol\Omega$.  The initial  functions $u_0\in C(\ol\Omega)$ and $v_0\in C(\ol\Omega)$ satisfy $u_0\geq,\not\equiv 0$ and $v_0\geq,\not\equiv 0$ on $\ol\Omega$.
	
\begin{itemize}
\item[{\rm (1)}] The problem \eqref{sp3} has a unique globally-defined  solution $(u(x,t),v(x,t))$ satisfying $u(x,t)>0$, $v(x,t)>0$ for $(x,t)\in\ol\Omega\times(0,\infty)$.

\item[{\rm (2)}] For any given small $\epsilon_1>0$, there exist  a constant  $T_1>0$ determined by $\epsilon_1$ and a constant $\epsilon_2\in (0,\epsilon_1]$ depending on initial functions such that
\begin{align}\label{2.3}
\epsilon_2\leq u(x,t), v(x,t)\leq \bar a+\epsilon_1,\  \ \forall\ x\in\ol\Omega,\ t\geq T_1,
\end{align}
which  implies that the problem \eqref{sp3} is permanent. Moreover the problem \eqref{sp3} has a positive steady state solution $(u_*(x),v_*(x))$ lying in $[\epsilon_2,\bar a+\epsilon_1]\times [\epsilon_2,\bar a+\epsilon_1]$.
\item[{\rm (3)}] There exits a constant $C=C(\epsilon_2)>0$ such that
\bes\label{2.4}
\max_{t\geq T_1}\|u(\cdot,t)\|_{C^{2+\alpha}(\ol\Omega)}, \max_{t\geq T_1}\|v(\cdot,t)\|_{C^{2+\alpha}(\ol\Omega)}
\leq C.
\ees
\end{itemize}
\end{proposition}
\begin{proof}
(1)  We will use the upper and lower solutions method to prove the existence and uniqueness of positive solution of problem \eqref{sp3}. Clearly,  the problem \eqref{sp3} is a mixed quasi-monotone system in the domain $\{u>0,\,v\geq0\}$. Denote
  \[M=\max\left\{\bar a, \ \max_{x\in\ol\Omega}u_0(x), \ \max_{x\in\ol\Omega}v_0(x)\right\}.\]
 Let $\underline v(x,t)=0$, $\bar u(x,t)=\bar v(x,t)\equiv M$, and let $\underline u(x,t)$ be the unique positive solution of
\begin{equation*}
\begin{cases}
u_t=d_1(x)\Delta u+u\lf(a(x)-u-\ds\frac{b M}{1+{r} u}\rr), & x\in \Omega,\; t>0,\\
\ds \f{\partial u}{\partial\nu}=0,& x\in \partial\Omega,\;
t>0,\\
u(x,0)=u_0(x),& x\in \Omega.
\end{cases}
\end{equation*}
Then $(\bar u(x,t),\bar v(x,t))$ and  $(\underline u(x,t),\underline v(x,t))$ are a pair of coupled ordered upper and lower solutions of the problem \eqref{sp3}.  Hence \eqref{sp3} has a unique global solution $(u(x,t),v(x,t))$  satisfying
\bes\label{2.5a}
0<\underline u(x,t)\leq u(x,t)\leq M,\quad 0\leq v(x,t)\leq M,\quad \forall\ x\in\ol\Omega,\ t\geq 0.
\ees
Moreover, by the strong maximum principle we also have $v(x,t)>0$ for $x\in\ol\Omega$ and $t>0$.

(2) From the first equation of \eqref{sp3},  $u(x,t)$ satisfies
\bess
\left\{\begin{array}{ll}
\ds u_t \leq d_1(x) \Delta u+u\left(\max_{x\in\bar{\Omega}}a(x)-u\right),& x \in \Omega,\ t>0,\\[1mm]
	\dd\frac{\partial u}{\partial \nu}=0,\ \ &x\in\partial\Omega,\ t\ge 0,\\[1mm]
	u(x,0)=u_0(x)>0,\ \ &x\in \Omega.
\end{array}\right.\eess
It is deduced by the comparison principle of parabolic equations that
\bess
 \dd\limsup_{t\to\infty}\max_{\ol\Omega}u(x,t)\leq \max_{x\in\bar{\Omega}}a(x)=\bar a.
 \eess
 Thus, for any given $\varepsilon>0$, there is a $T>0$ such that
$u(x,t)<\bar a+\varepsilon$ for $x\in\ol\Omega,\ t\geq T$. From the second equation of \eqref{sp3}, $v(x,t)$ satisfies
\bess
v_t\leq  d_2 \Delta v+\mu v\big(1-v/(\bar a+\varepsilon)\big),  \quad x>\Omega,\ t>T.
\eess
Thanks to   the boundary condition $\frac{\partial v}{\partial \nu}=0$, we could use the comparison principle of parabolic equations  to conclude  that  $\dd\limsup_{t\to\infty}\max_{\ol\Omega}v(x,t)\leq \bar a+\varepsilon$. The arbitrariness of $\varepsilon$ implies
\bess
\dd\limsup_{t\to\infty}\max_{\ol\Omega}v(x,t)\leq \bar a.
\eess
Hence, for a small fixed $\epsilon_1>0$ satisfying
\begin{align}\label{2.6a}
b<(\underline a-\epsilon_1)/(\bar a-\epsilon_1),
\end{align}
 there exists a $T_1>0$  such that the following estimates hold
\bess
u(x,t), v(x,t)\leq \bar a+\epsilon,\  \ \   x\in\ol\Omega,\ t\geq T_1.
\eess
Since $u(x,T_1),v(x,T_1)>0$ on $\ol\Omega$, we can choose a small constant $\epsilon_2>0$ belonging to $(0,\epsilon_1]$  such that $u(x,T),v(x,T)>\epsilon_2$ on $\ol\Omega$. Denote
\bess
\bar u_1:=\bar a+\epsilon_1,\ \ \underline u_1:=\epsilon_2,\ \  \bar v_1:={\bar a}+\epsilon_1,\ \ \underline v_1:=\epsilon_2.
\eess
Then from $\epsilon_2\leq \epsilon_1$ and \eqref{2.6a}, we obtain that
\begin{equation}\label{2.7a}
\begin{cases}
\ds a(x)-\bar u_1-\frac{b \underline v_1 }{1+{r} \bar u_1}\leq a(x)-(\bar a+\epsilon_1)<0, \\
\ds a(x)-\underline u_1-\frac{b \bar v_1 }{1+{r} \bar u_1}\geq a(x)-\epsilon_1-b(\bar a+\epsilon_1)>0,\\
1-{\underline v_1}/{\underline u_1}= 1-{\bar v_1}/{\bar u_1}=0,
\end{cases}
\end{equation}
which indicates that   $(\bar u_1, \underline u_1,\bar v_1, \underline v_1)$ is a pair of coupled ordered upper and lower solutions of the problem \eqref{sp3} with initial density $(u(x,T_1),v(x,T_1))$. Hence \eqref{2.3} holds.
A simple calculation shows that $(\bar u_1, \underline u_1,\bar v_1, \underline v_1)$ is also the coupled ordered upper and lower solutions of the problem \eqref{esp3}. Thus  the problem \eqref{esp3} has a positive  solution $(u_*,v_*)$ in the region  $[\underline u_1,\bar u_1]\times [\underline u_1,\bar u_1]$. 
For (3),
recalling that $u(x,t),v(x,t)>\epsilon_2$ for $x\in\bar{\Omega}$, $t\geq T_1$, we could show    \eqref{2.4}  by the similar arguments as \cite[Theorem 2.1]{wmx2016jfa}. The proof is completed.
\end{proof}

\subsection{Estimates for positive solutions and steady state solutions}

From Proposition \ref{th2.1}, under the assumption \eqref{2.2} every positive solution of \eqref{sp3} has a positive lower  bound which may depend on its initial value. In this subsection,  a uniform lower bound for positive solutions of \eqref{sp3} is obtained. Moreover, by an iterating process using the idea of \cite{pao2002},  we obtain more accurate estimates for positive solutions and steady solutions of problem \eqref{sp3}.

Denote
\bes\label{2.7}
\bar u_1=\bar v_1:=\bar a+\epsilon_1,\ \ \underline u_1=\underline v_1:=\epsilon_2,
 \ \ Q:=[\underline u_1,\bar u_1]\times [\underline v_1,\bar v_1].\ees
where $\epsilon_1$ and $\epsilon_2$ are given by Proposition \ref{th2.1}. It is clear that
\begin{equation}\label{2.9}
\begin{cases}
f_y(y,u,v)\geq 0,\ f_v(y,u,v)\leq 0,\ g_u(y,u,v)\geq 0, &y,u,v> 0,\\
|f(y,u_1,v_1)-f(y,u_2,v_2)|\leq K(|u_1-u_2|+|v_1-v_2|), &  y\geq 0, (u,v)\in Q,\\
|g(u_1,v_1)-g(u_2,v_2)|\leq K(|u_1-u_2|+|v_1-v_2|),  &  (u,v)\in Q,
\end{cases}
\end{equation}
for some $K>0$. With $\underline u_1$, $\bar u_1$, $\underline v_1$ and $\bar v_1$ given by \eqref{2.7},  we define the following iterative sequences:
 \bess
\left\{\begin{array}{ll}
\ds\bar u_{i+1}=\bar u_i+\frac{1}{K}f(\bar a,\bar u_i,\underline v_i),\ \ \underline  u_{i+1}=\underline  u_i+\frac{1}{K}f(\underline a,\underline u_i,\bar v_i),\ \ &i=1,2,\cdots,\\[1mm]
\ds\bar v_{i+1}=\bar v_i+\frac{1}{K}g(\bar u_i,\bar v_i),\ \ \underline  v_{i+1}=\underline v_i+\frac{1}{K}g(\underline u_i,\underline v_i),\ \ &i=1,2,\cdots,
 \end{array}\right.\eess

The iterative sequence defined above satisfy the following monotonicity and convergence properties.
\begin{lemma}\label{lemma2.2} Suppose that \eqref{2.2} holds.
	
{\rm (i)} The sequences of constants $\{\bar u_i\}_{i}^{\yy}$, $\{\underline u_i\}_{i=1}^{\yy}$, $\{\bar v_i\}_{i=1}^{\yy}$ and $\{\underline u_i\}_{i=1}^{\yy}$ satisfy
\bess
\left\{\begin{array}{ll}
0<\underline  u_1\leq\cdots \leq \underline u_i\leq \underline u_{i+1}\leq \cdots\leq \bar u_{i+1}\leq \bar u_{i}\cdots \leq \bar u_1,\\
0<\underline v_1\leq \cdots \leq \underline v_i\leq \underline v_{i+1}\leq \cdots\leq \bar v_{i+1}\leq  \bar v_{i}\cdots \leq \bar v_1.
 \end{array}\right.\eess

{\rm (ii)}  Denote
$\bar u_{\yy}:=\dd\lim_{i\to\yy} \bar u_i$, $ \ds \underline u_{\yy}:=\lim_{i\to\yy} \underline  u_i$, $\ds \bar v_{\yy}:=\lim_{i\to\yy} \bar v_i$ and $\ds \underline v_{\yy}:=\lim_{i\to\yy} \underline  v_i$.
Then
\begin{align}\label{2.10}
f(\bar a,\bar u_{\yy},\underline v_{\yy})=0,\;  f(\underline a,\underline u_{\yy},\bar v_{\yy})=0,\quad
g(\bar u_{\yy},\bar v_{\yy})=0,\;  g(\underline u_{\yy},\underline v_{\yy})=0.
\end{align}
\end{lemma}
\begin{proof}
	(i) We prove the monotonicity of $\underline u_i$, $\bar u_i$, $\underline v_i$ and $\bar v_i$ with respect to $i$ inductively. First,  we show that
	\begin{align}\label{2.12a}
	\underline u_1\leq \underline u_2\leq \bar u_2\leq\bar u_1,\ \
	\underline v_1\leq \underline v_2\leq  \bar v_2\leq\bar v_1.
	\end{align}
From \eqref{2.7a},
	 \begin{equation}\label{2.13a}
\begin{cases}
\underline u_1\leq \bar u_1,\ \underline v_1\leq \bar v_1,\\
f(\bar a,\bar u_1,\underline v_1)\leq 0,\ f(\underline a,\underline u_1,\bar v_1)\geq 0,\  g(\bar u_1,\bar v_1)\leq 0, \ g(\underline u_1,\underline v_1)\geq 0.
\end{cases}
	 \end{equation}
Then the definitions of $\underline u_2$, $\bar u_2$,  $\underline v_2$ and $\bar v_2$  give $\underline u_1\leq \underline u_2$, $\bar u_2\leq \bar u_1$, $\underline v_1\leq \underline v_2$ and  $\bar v_2\leq \bar v_1$. From $\bar a\geq \underline a$ and  \eqref{2.9}, we derive
\begin{align*}
\bar u_{2}-\underline  u_{2}&=\bar u_1+\frac{1}{K}f(\bar a,\bar u_1,\underline v_1)-\underline  u_1-\frac{1}{K}f(\underline a,\underline u_1,\bar v_1)\\
&\geq \bar u_1-\underline  u_1+\frac{1}{K}f(\underline a,\bar u_1,\bar v_1)-\frac{1}{K}f(\underline a,\underline u_1,\bar v_1)\geq 0.
\end{align*}
 Similarly, we obtain $\bar v_{2}\geq \underline  v_{2}$. Therefore, \eqref{2.12a} holds.  Suppose that for $i\in {\mathbb N}$, we have
 \begin{align*}
 \underline u_1\leq \underline u_2\leq\cdots \leq \underline u_i\leq  \bar u_{i}\cdots \leq \bar u_2\leq\bar u_1,\ \
 \underline v_1\leq \underline v_2\leq\cdots \leq \underline v_i\leq   \bar v_{i}\cdots \leq \bar v_2\leq\bar v_1.
 \end{align*}
From  \eqref{2.9}, it follows
\begin{align*}
\bar u_{i+1}-\underline  u_{i+1}&=\bar u_i+\frac{1}{K}f(\bar a,\bar u_i,\underline v_i)-\underline  u_i-\frac{1}{K}f(\underline a,\underline u_i,\bar v_i),\\
&\geq \bar u_i-\underline  u_i+\frac{1}{K}f(\underline a,\bar u_i,\bar v_i)-\frac{1}{K}f(\underline a,\underline u_i,\bar v_i)\geq 0,\\
\bar u_{i+1}-\bar u_{i}&=\bar u_i+\frac{1}{K}f(\bar a,\bar u_i,\underline v_i)-\bar u_{i-1}-\frac{1}{K}f(\bar a,\bar u_{i-1},\underline v_{i-1})\\
&\leq \bar u_i-\bar u_{i-1}+\frac{1}{K}f(\bar a,\bar u_i,\underline v_{i-1})-\frac{1}{K}f(\bar a,\bar u_{i-1},\underline v_{i-1})\leq 0.
\end{align*}
Similarly, we can show that $\underline u_{i+1}\geq \underline u_{i}, \underline v_{i+1}\geq \underline v_{i}, \bar v_{i+1}\geq \underline  v_{i+1}$ and $\bar v_{i+1}\leq  \bar v_{i}$. Therefore the conclusion in (i) holds.

(ii) The formulas  in  (i) imply that the  sequences  $\{\bar u_i\}_{i}^{\yy}$, $\{\underline u_i\}_{i=1}^{\yy}$, $\{\bar v_i\}_{i=1}^{\yy}$ and $\{\underline u_i\}_{i=1}^{\yy}$  converge to some constants, respectively. Then \eqref{2.10} follows from the definitions of $\bar u_i, \bar v_i, \underline u_i$ and $\underline v_i$. The proof is completed.
 \end{proof}

The above lemma states that the system of equations \eqref{2.10}
admits a positive solution. We next show that the positive solution of \eqref{2.10} is unique, which in fact is the conclusion of Theorem \ref{thm1} (i).
\begin{proof}[Proof of Theorem \ref{thm1} (i)]
	From the definition of $f$ and $g$, every possible positive solution $(\underline u,\bar u,\underline v,\bar v)$ of  \eqref{2.10}  satisfies   $\bar v=\bar u$, $\underline v=\underline u$ and
	\bes\label{2.12}
	(\bar a- \bar u)(1+{r} \bar u)=b\underline u,\ \ \ \ 	(\underline a-\underline u)(1+{r} \underline u)=b\bar u,
	\ees
	which is equivalent to
	\begin{align}\label{2.17a}
	h(\underline u)=0 \ \ {\rm for }\ \ \underline u\in (0,\underline a)\ \ {\rm and}\ \ \bar u=(\underline a-\underline u)(1+{r} \underline u)/b,
	\end{align}
	where
	\begin{align*}
	h(\tau):=&[b\bar a-(\underline a-\tau)(1+{r} \tau)][b+{r} (\underline a-\tau)(1+{r} \tau)]-b^3\tau\\
	=&b^2\bar a+(b\bar a{r} -b)(\underline a-\tau)(1+{r}\tau)-{r} (\underline a-\tau)^2(1+{r}\tau)^2-b^3\tau.
	\end{align*}
	
	\noindent	{\bf Case 1}. $r>0$.

	Making use of $b \bar a< \underline a$ (see \eqref{2.2}) and $\underline a\leq \bar a$, we get
	\begin{align*}
	&h(0)=(b\bar a- \underline a)(b+\underline a{r})<0,\\
	&h(\underline a)=b^2\bar a -b^3\underline a=b^2(\bar a-b \underline a)>0.
	\end{align*}	
	Note  $\ds\lim_{|\tau|\to\yy}h(\tau)=-\yy$. We see that the equation $h(\tau)=0$ in $\tau\in [0,\underline a]$ either admits a unique zero or has two or three zeros. In the later case, $h$ satisfies
	\begin{align}\label{2.18}
	h'(\tau)\geq  0\ \ \ {\rm for\ all}\ \tau\leq 0,
	\end{align}
	which will be excluded in the following discussion.
	
	Direct calculation yields
	\begin{align*}
	h'(\tau)=&(b\bar a{r}-b)(\underline a{r}-1-2{r} \tau )-2{r}(\underline a-\tau)(1+{r} \tau)(\underline a{r}-1-2{r} \tau )-b^3\nonumber\\
	=&[(b\bar a{r}-b)-2{r}(\underline a-\tau)(1+{r} \tau)](\underline a{r}-1-2{r} \tau )-b^3.
	\end{align*}
	If $\underline a{r}\geq 1$,  then from $b\bar a<\underline a$,
	\begin{align*}
	&h'(0)=(b\bar a{r}-b-2{r} \underline a)(\underline a{r}-1)-b^3< (-b-\underline a{r})(\underline a{r}-1)\leq 0.
	\end{align*}
	On the other hand, if $\underline a{r}< 1$, then $(\underline a{r}-1)/(2{r})<0$ and
	\begin{align*}
	h\lf(\frac{\underline a{r}-1}{2{r}}\rr)=-b^3<0.
	\end{align*}
	Thus, \eqref{2.18} is impossible for any ${r}>0$. Consequently, $h(\tau)=0$ has only one zero in $\tau\in [0,\underline a]$.
	
	\noindent	{\bf Case 2}. $r=0$.
	
	Clearly, $h(\tau)=[b\bar a-(\underline a-\tau)]b-b^3\tau=b[b\bar a-\underline a+(1-b^2)\tau]$, and from \eqref{2.17a},
	\bes\label{2.14a}
	\underline u_{\yy}=\underline u_{\yy}=\frac{\underline a-b\bar a}{1-b^2},\ \bar u_{\yy}=\bar u_{\yy}=\frac{\bar a-b\underline a}{1-b^2}.
	\ees
	The proof is completed.
\end{proof}

We call that   $(\underline u_s(x), \bar u_s(x), \underline v_s(x), \bar v_s(x))$ is a pair of  quasi-solution of problem \eqref{esp3} if $(\underline u_s(x), \bar u_s(x), \underline v_s(x), \bar v_s(x))$ satisfies $\underline u_s(x)\leq \bar u_s(x)$, $\underline v_s(x)\leq \bar v_s(x)$ and
\begin{equation*}
\begin{cases}
-d_1(x)\Delta \bar u_s=f(x,\bar u_s(x),\underline v_s(x)), & x\in \Omega,\\
-d_1(x)\Delta \underline u_s=f(x,\underline u_s(x),\bar v_s(x)), & x\in \Omega,\\
\ds -d_2(x)\Delta \bar v_s=g(\bar u_s,\bar u_s), & x\in \Omega,\\
\ds -d_2(x)\Delta \underline v_s=g(\underline u_s,\underline u_s), & x\in \Omega,\\
\ds \f{\partial \bar u_s}{\partial\nu}=\ds \f{\partial \underline u_s}{\partial\nu}=\ds \f{\partial \bar v_s}{\partial\nu}=\ds \f{\partial \underline v_s}{\partial\nu}=0,& x\in \partial\Omega.
\end{cases}
\end{equation*}

\begin{proposition}\label{lemma2.3} Suppose \eqref{2.2} holds.
	
{\rm (i)} Let $(\underline u_s(x), \bar u_s(x), \underline v_s(x), \bar v_s(x))$ with $\underline u_s(x)$, $\bar u_s(x)\in [\underline u_1,\bar v_1]$, and $\underline v_s(x)$, $\bar v_s(x)\in [\underline v_1,\bar u_1]$  be a positive quasi-solution of problem \eqref{esp3}.   Then
\bes\label{2.14}
\underline u_{\yy}\leq\underline u_s(x)\leq \bar u_s(x)\leq \bar u_{\yy},\ \ \underline v_{\yy}\leq\underline v_s(x)\leq \bar v_s(x)\leq \bar v_{\yy},
\ees
where $\underline u_{\yy}, \underline v_{\yy}$, $\bar u_{\yy}, \bar v_{\yy}$ are defined by Lemma \ref{lemma2.2}.

{\rm (ii)} Let $(u(x,t),v(x,t))$ be the positive solution of problem \eqref{sp3}, and let $(u_*(x),v_*(x))$ be a positive steady state solution of \eqref{sp3}. Then the following estimates hold
 \bes\label{2.15}
 \left\{\begin{array}{ll}
\ds \underline u_{\yy}\leq \liminf_{t\to\yy}u(x,t)\leq \limsup_{t\to\yy}u(x,t)\leq \bar u_{\yy},\ \
\underline u_{\yy}\leq u_*(x)\leq \bar u_{\yy},\\
\ds \underline v_{\yy}\leq \liminf_{t\to\yy}v(x,t)\leq \limsup_{t\to\yy}v(x,t)\leq \bar v_{\yy}, \ \ \underline v_{\yy}\leq v_*(x)\leq \bar v_{\yy}.
  \end{array}\right.\ees
\end{proposition}
\begin{proof}
(i) To prove \eqref{2.14}, it is sufficient to show that
\bes\label{2.16}
\underline u_{i}\leq\underline u_s(x)\leq \bar u_s(x)\leq \bar u_{i},\ \ \underline v_{i}\leq\underline v_s(x)\leq \bar v_s(x)\leq \bar v_{i}, \ \ \forall \ i=1,2,\cdots
\ees
The proof is by induction on $i$. 
Since $\bar u_1=\bar v_1=\bar{a}+\epsilon_1$ and $\underline u_1=\underline v_1=\epsilon_2$,  the inequalities in \eqref{2.16} hold for $i=1$. Assuming the inequalities in \eqref{2.16} hold for $i\leq j_0$ where $j_0\geq 2$ is an integer, we will prove it for $i=j_0+1$. Making use of \eqref{2.9}, we deduce that
\begin{align*}
&-d_1(x) \Delta \bar u_{j_0+1}+K\bar u_{j_0+1}-Ku_s-f(x,\bar u_s,\underline v_s)\\
= & K\bar u_{j_0+1}-Ku_s-f(x,\bar u_s,\underline v_s)
=  K\bar u_{j_0}+f(\bar a,\bar u_{j_0},\underline v_{j_0})-Ku_s-f(x,\bar u_s,\underline v_s)\\
\geq & K\bar u_{j_0}-Ku_s+f(x,\bar u_{j_0},\underline v_s)-f(x,\bar u_s,\underline v_s)\geq 0.
\end{align*}
Denote $w(x)=\bar u_{j_0+1}-\bar u_s(x)$. Then $w$ satisfies $-d_1(x) \Delta w+K w\geq 0$ in $\Omega$ with  $\partial_{\nu}  w=0$ on $\partial\Omega$. It is derived by the maximum  principle of elliptic equations that $w\geq 0$. Ans so $\bar u_{j_0+1}\geq \bar u_s(x)$ on $\bar{\Omega}$. Similarly, we can prove that $\underline u_{j_0+1}\leq\underline u_s(x)$, $\underline v_{j_0+1}\leq\underline v_s(x)$ and $\bar v_s(x)\leq \bar v_{j_0+1}$. Thus the inequalities in \eqref{2.16} hold.

(ii)   If  the initial densities $u_0, v_0$ lie in the region $[\epsilon_2, \bar a+\epsilon_1]$, by  Theorem 3.2 in \cite{pao1997},  the solution $(u(x,t),v(x,t))$ satisfies the estimates in \eqref{2.15}. Recalling  \eqref{2.3}, we obtain that  any positive solution $(u(x,t),v(x,t))$ of \eqref{sp3} satisfies the estimates in \eqref{2.15}.

For  any positive steady state solution $(u_*(x),v_*(x))$  of \eqref{sp3},  $(u_*(x),u_*(x),v_*(x),v_*(x))$ is a pair of positive quasi-solution of problem \eqref{esp3}. Combining this fact with \eqref{2.3} and \eqref{2.14}, we obtain the estimates for  $(u_*(x),v_*(x))$ in \eqref{2.15}. The proof is completed.
\end{proof}

We remark that the conclusions of Proposition {\rm \ref{lemma2.3}} hold for some general functions $f$ and $g$ satisfying   \eqref{2.9} and \eqref{2.13a}.

\begin{proof}[Proof of Theorem \ref{thm1} (ii)]  It is clear that Theorem \ref{thm1} (ii) follows directly from Proposition
\ref{lemma2.3}.
\end{proof}

\subsection{Global stability of positive steady state solution}
To prove the global stability of positive steady state solution of \eqref{sp3}, we need the following two lemmas.
\begin{lemma}{\rm(\hspace{-.1mm}\cite[Theorem 1.1]{wmx2018aml} or \cite[Lemma 2.2]{nsw1})}\label{lemma2.5} Let $\delta> 0$ be a constant, and let  the two functions $\psi,h\in C([\delta,\yy))$ satisfy $\psi(t)\geq0$ and $\ds\int_{\delta}^{\yy} h(t)dt<\yy$, respectively.  Assume that $\varphi\in C^1([\delta,\yy))$ is bounded from below and satisfies
	\bess
	\varphi'(t)\leq-\psi(t)+h(t)\ \ \ \text{in}\, \ [\delta,\yy).
	\eess
	If one of the following conditions holds:
	\begin{itemize}
		\item[{\rm(i)}]  $\psi$ is uniformly continuous in $[\delta,\yy)$,
		\item[{\rm(ii)}] $\psi\in C^1([\delta,\yy))$  and $\psi'(t)\leq K$ in $[\delta,\yy)$ for some constant $K >0$,
		\item[{\rm (iii)}] $\psi\in C^{\beta}([\delta,\yy))$ with $0<\beta<1$, and for $\tau>0$  there exists $K>0$ just depending on $\tau$ such that   $\|\psi\|_{C^{\beta}([x,x+\tau])}\leq K$ for all $x\geq \delta$,
	\end{itemize}
	then $\dd\lim_{t\to\yy}\psi(t)=0$.
\end{lemma}

\begin{lemma}{\rm \cite[Lemma 2.3]{nsw1}}\label{lemma2.6}
	Let $w,\,w_*\in C^{2}(\ol\Omega)$ be two positive functions. If   $\ds\f{\partial w}{\partial\nu}=0$ and $\ds\f{\partial w_*}{\partial\nu}=0$ on $\partial\Omega$, then
	\begin{align}\label{2.17}
	\int_{\Omega}\frac{w_*[w-w_*]}{w}\bigg(\Delta w-\frac{w}{w_*}\Delta w_*\bigg){\rm d}x
	\leq-\int_{\Omega} w^2\Big|\nabla \frac{w_*}{w}\Big|^2{\rm d}x\leq0.
	\end{align}
\end{lemma}

With the help of  the above results, we now show  the global stability of positive steady state solution of \eqref{sp3} using Lyapunov functional method.

\begin{proof}[Proof of Theorem \ref{thm1} (iii)]
Let $(u(x,t),v(x,t))$ be the solution of \eqref{sp3}.  Define  a function $G: [0,\yy)\to \R$ by
	\bess
	G(t):=\int_{\Om}\int_{u_*(x)}^{u(x,t)} \frac{u_*(x)}{d_1(x)}\frac{s-u_*(x)}{s} \text{d}s\text{d}x+\eta\int_{\Om}\int_{v_*(x)}^{v(x,t)} \frac{v_*(x)}{d_2(x)}\frac{s-v_*(x)}{s} \text{d}s\text{d}x
	\eess
	with $\eta>0$ to be determined later.  Then $G(t)\geq 0$ for $t\geq 0$. Making use of \eqref{2.17},  we deduce
	\begin{align*}
		\frac{dG(t)}{dt}=&\int_{\Omega}\frac{u_*(u-u_*)}{d_1u}u_t\text{d}x+\eta\int_{\Omega}\frac{v_*(v-v_*)}{d_2v}v_t\text{d}x\nonumber\\
	=&\int_{\Omega}\lf(\frac{u_*(u-u_*)}{d_1u}\lf[d_1\Delta u+f(x,u,v)\rr]+\eta\frac{v_*(v-v_*)}{d_2v}\lf[d_2\Delta v+g(u,v)\rr]\rr)\text{d}x\nonumber\\[1.5mm]
	=&\int_{\Omega}\frac{u_*(u-u_*)}{d_1u}\lf(d_1\Delta u+f(x,u,v)
	-\frac{u}{u_*}d_1\Delta u_*-\frac{u}{u_*}f(x,u_*,v_*)\rr)\text{d}x\nonumber\\[1.5mm]
	&+\eta\int_{\Omega}\frac{v_*(v-v_*)}{d_2v}\lf(d_2\Delta v+g(u,v)
	-\frac{v}{v_*}d_2\Delta v_*-\frac{v}{v_*}g(u_*,v_*)\rr)\text{d}x\nonumber\\[1.5mm]
	=&\int_{\Omega}\lf[\frac{u_*(u-u_*)}{u}\lf(\Delta u-\frac{u}{u_*}
	\Delta u_*\rr)+\frac{u_*(u-u_*)}{d_1}\lf(\frac{f(x,u,v)}{u}-\frac{f(x,u_*,v_*)}{u_*}\rr)\rr]\text{d}x\nonumber\\[1.5mm]
	&+\eta\int_{\Omega}\lf[\frac{v_*(v-v_*)}{v}\lf(\Delta v-\frac{v}{v_*}
	\Delta v_*\rr)+\frac{v_*(v-v_*)}{d_2}\lf(\frac{g(u,v)}{v}-\frac{g(u_*,v_*)}{v_*}\rr)\rr]\text{d}x\nonumber\\[1.5mm]
	&\leq\int_{\Omega}\frac{u_*(u-u_*)}{d_1}\lf(\frac{f(x,u,v)}{u}-\frac{f(x,u_*,v_*)}{u_*}\rr)\text{d}x\nonumber\\
	&+\eta\int_{\Omega}\frac{v_*(v-v_*)}{d_2}\lf(\frac{g(u,v)}{v}-\frac{g(u_*,v_*)}{v_*}\rr)\text{d}x.
	\end{align*}
	By the definition of $f$ and $g$ in \eqref{2.7}, we derive
	\begin{align*}
	\frac{dG(t)}{dt}\leq &\int_{\Omega}\frac{u_*(u-u_*)}{d_1}\lf(-u-\frac{bv}{1+{r} u}+u_*+\frac{bv_*}{1+{r} u_*}\rr)\text{d}x\\
	&+\int_{\Omega}\eta\frac{v_*(v-v_*)}{d_2}\mu\lf(-\frac{v}{u}+\frac{v_*}{u_*}\rr)\text{d}x\\
	=&\int_{\Omega}\frac{-u_*[(1+{r} u)(1+{r} u_*)-b {r} v_*](u-u_*)^2-bu_* (1+{r} u_*)(u-u_*)(v-v_*)}{d_1(1+{r} u)(1+{r} u_*)}\text{d}x\\
	&+\int_{\Omega}\frac{\eta\mu v_*^2(u-u_*)(v-v_*)-\eta\mu u_*v_*(v-v_*)^2}{d_2uu_*}\text{d}x\\
	=&\int_{\Omega}\frac{E}{d_1d_2u u_*(1+{r} u)(1+{r} u_*)}\text{d}x,
	\end{align*}
	with
	\begin{align*}
	E:=&-d_2u u_*^2[(1+{r} u)(1+{r} u_*)-b{r} v_*](u-u_*)^2-b d_2u u_*^2(1+{r} u_*)(u-u_*)(v-v_*)\\
	&+d_1(1+{r} u)(1+{r} u_*) [\eta\mu v_*^2(u-u_*)(v-v_*)-\eta\mu u_*v_*(v-v_*)^2]\\
	=& A (u-u_*)^2+B(u-u_*)(v-v_*) +C (v-v_*)^2,
	\end{align*}
	where
	\begin{align*}
	A:=&-d_2u u_*^2[(1+{r} u)(1+{r} u_*)-b{r} v_*],\ C:=-d_1\eta\mu u_*v_*(1+{r} u)(1+{r} u_*), \\
	B:=&-b d_2u u_*^2(1+{r} u_*)+d_1\eta\mu v_*^2(1+{r} u)(1+{r} u_*).
	\end{align*}
	
	Next we choose a suitable $\eta>0$ such that $2\sqrt{AC}>|B|$, which then yields
	\begin{align}\label{2.20}
	\frac{dG(t)}{dt}\leq -\int_{\Omega}\frac{\delta(u-u_*)^2+\delta(v-v_*)^2}{d_1d_2uu_*(1+{r} u)(1+{r} u_*)}\text{d}x=:\psi(t)\leq 0,
	\end{align}
	for some  $0<\delta\ll 1$. 	Denote $\bar d_i=\max_{x\in\ol\Omega} d_i(x)$, $\underline d_i=\min_{x\in\ol\Omega} d_i(x)$ and
	\begin{align*}
	\eta =	\sqrt{\frac{\underline d_2 \bar d_2(\underline u_\yy-\epsilon) (\bar u_\yy-\epsilon)\underline u_\yy^3\bar u_\yy}{\underline d_1\bar d_1\underline v_\yy \bar v_\yy^3}} \frac{b}{\mu[1+{r} (\underline u_\yy-\epsilon)]}
	\end{align*}
	for some small $\epsilon>0$.   From \eqref{2.15},   there exists $T>1$ such that $u(x,t)\geq \underline u_{\yy}-\epsilon$ and $v(x,t)\geq \underline v_{\yy}-\epsilon$   for all $t\geq T$.  A simple calculation gives
	\begin{align*}
	&2\sqrt{AC}-|B|\geq 2\sqrt{AC}-[b d_2u u_*^2(1+{r} u_*)+d_1\eta\mu v_*^2(1+{r} u)(1+{r} u_*)]\\
	= &2\sqrt{d_1d_2\eta \mu uu_*^3v_*(1+{r} u)(1+{r} u_*)[(1+{r} u)(1+{r} u_*)-b{r} v_*]}\\
	&-[b d_2u u_*^2(1+{r} u_*)+d_1\eta\mu v_*^2(1+{r} u)(1+{r} u_*)]\\
	=&(1+{r} u)(1+{r} u_*)\sqrt{d_1d_2uu_*^3v_*}
	\bigg[2 \sqrt{\eta \mu-\frac{b\eta \mu{r} v_*}{(1+{r} u)(1+{r} u_*)}}\\
	&-\lf(b\sqrt{\frac{d_2uu_*}{d_1v_*}} \frac{1}{1+{r} u}+\eta\mu\sqrt{\frac{d_1v_*^3}{d_2uu_*^3}} \rr)\bigg]\\
	=&:E_1\lf[2 \sqrt{\eta \mu-\frac{b\eta \mu{r} v_*}{(1+{r} u)(1+{r} u_*)}}-\lf(b\sqrt{\frac{d_2uu_*}{d_1v_*}} \frac{1}{1+{r} u}+\eta\mu\sqrt{\frac{d_1v_*^3}{d_2uu_*^3}} \rr)\rr].
	\end{align*}
	Taking advantages of  \eqref{2.15}, $\underline v_{\yy}=\underline u_{\yy}$, $\bar v_{\yy}=\bar u_{\yy}$ and the definition of $\eta$,
	we derive that  for $t\geq T$,
	\begin{align*}
	&2\sqrt{AC}-|B|\geq E_1\Bigg[2 \sqrt{\eta \mu-\frac{b\eta \mu{r} \bar  v_\yy}{[1+{r} (\underline u_\yy-\epsilon)](1+{r} \underline u_\yy)}}\\
	&-\lf(b\sqrt{\frac{\bar d_2(\bar u_\yy-\epsilon)\bar u_\yy}{\underline d_1\underline v_\yy}} \frac{1}{1+{r} (\underline u_\yy-\epsilon)}+\eta\mu\sqrt{\frac{\bar d_1\bar v_\yy^3}{\underline d_2 (\underline u_\yy-\epsilon) \underline u_\yy^3}} \rr)\Bigg]\\
	=&E_1\lf(2 \sqrt{\eta \mu-\frac{b\eta \mu{r} \bar  v_\yy}{[1+{r} (\underline u_\yy-\epsilon)](1+{r} \underline u_\yy)}}-
2\sqrt{	\frac{b\eta\mu }{1+{r} (\underline u_\yy-\epsilon)}
	\sqrt{\frac{\bar d_1\bar d_2(\bar u_\yy-\epsilon)\bar u_\yy\bar v_\yy^3}{\underline d_1\underline d_2(\underline u_\yy-\epsilon) \underline u_\yy^3\underline v_\yy}}
}\rr)\\
=&E_1\lf(2 \sqrt{\eta \mu-\frac{b\eta \mu{r} \bar  u_\yy}{[1+{r} (\underline u_\yy-\epsilon)](1+{r} \underline u_\yy)}}-
2\sqrt{	\frac{b\eta\mu }{1+{r} (\underline u_\yy-\epsilon)}
	\sqrt{\frac{\bar d_1\bar d_2(\bar u_\yy-\epsilon)\bar u_\yy^4}{\underline d_1\underline d_2(\underline u_\yy-\epsilon) \underline u_\yy^4}}
}\rr).
	\end{align*}
Then $2\sqrt{AC}-B>0$ follows from \eqref{4.6} (with $\epsilon\to 0$) and $(\underline a-\underline u)(1+{r} \underline u)=b\bar u$ (See \eqref{2.12}). Thus \eqref{2.20} holds for $t\geq T$.
	
Next we show the global stability of the positive steady state solution $(u_*,v_*)$. By \eqref{2.4} and  the definition of $\psi(t)$, we see that  $|\psi'(t)|<C_1$ in $t\in[T,\yy)$ for some $C_1>0$. Then it follows from Lemma
\ref{lemma2.5} that
\bess
\lim_{t\to\yy}\psi(t)=-\int_{\Omega}\frac{\delta(u-u_*)^2+\delta(v-v_*)^2}{d_1d_2uu_*}\text{d}x=0.
\eess
Recalling that  $u(t,x)\geq \epsilon_2>0$ for $t\geq T_1$ by \eqref{2.3}, we have
\bes\label{2.19}
\lim_{t\to\yy}u(x,t)=u_*(x), \ \ \lim_{t\to\yy}v(x,t)=v_*(x) \ \ {\rm in} \ \  L^2(\ol\Omega).
\ees
The estimate \eqref{2.4} also implies that the set $\{u(\cdot,t):t\geq 1\}$ is relatively compact
in $C^2(\ol\Omega)$. Therefore, we may assume that
\[\|u(x,t_k)-\td u(x)\|_{C^2(\ol\Omega)},\|v(x,t_k)-\td v(x)\|_{C^2(\ol\Omega)}\to 0\ \ \ {\rm as}\ \ t_k\to\yy\]
for some functions $\td u,\td v\in C^2(\ol\Omega)$. Combining this with \eqref{2.19}, we could conclude that $\td u(x)\equiv u_*(x)$ and $\td v(x)\equiv v_*(x)$ for $x\in\ol\Omega$. Thus $\dd\lim_{t\to\yy}u(x,t)=u_*(x)$ and $\dd\lim_{t\to\yy}v(x,t)=v_*(x)$ in $C^2(\ol\Omega)$.
The proof is finished.
\end{proof}

\begin{remark}\label{end} The condition \eqref{4.6} for the global stability of positive steady state is an implicit one as
the quasi-steady state $(\bar u_\yy, \underline u_\yy, \bar v_\yy, \underline v_\yy)$ cannot be solved explicitly except when $r=0$ (see \eqref{ss}). We observe that \eqref{4.6} holds in the following cases:
	
	\begin{itemize}
		\item[{\rm (1)}]   $b>0$ is sufficiently small. In fact,
		from \eqref{2.12} we see $\bar u_\yy \approx\ds\bar a-\frac{\underline u_\yy}{(1+{r} \bar u_\yy)}b\to \bar a$ and $\underline u_\yy \approx \ds\underline a-\frac{\bar u_\yy}{(1+{r} \underline u_\yy)}b\to \underline a$ as $b\to 0$ since $\bar u_\yy$, $\underline u_\yy\in [0,\bar a]$. Hence,
		\begin{align*}
		(1+2r \underline u_\yy-r \underline a)\lf(\frac{\underline u_\yy}{\bar u_\yy}\rr)^{5/2}\to (1+r \underline a)  \lf({\underline a}/{\bar a}\rr)^{5/2}\ \ \ {\rm as}\ b\to 0,
		\end{align*}
		which immediately implies that \eqref{4.6} is satisfied for small $b>0$.
		\item[{\rm (2)}]  $a_A=\bar a-\underline a$ is sufficiently small. For any $M>1$, there exists $\tilde a>0$ such that when $0<a_A<\tilde a$, we have $\bar u_\yy/\underline u_\yy<M$. Then \eqref{4.6} holds if $b$ satisfies
\bes\label{4.66}
 b<(1+r \underline a)\lf[\frac{\min d_1(x)\min d_2(x)}
{\max d_1(x)\max d_2(x)} \rr]^{1/2} M^{-5/2}.
\ees
	\end{itemize}
\end{remark}

		\bibliographystyle{plain}
	\bibliography{NSW-spatial-predator-prey}
	
\end{document}